\documentclass[12pt,fleqn]{article}
\usepackage{amsfonts,amsthm,amssymb,amsmath,amscd}
\usepackage{wasysym}
\usepackage[all]{xy}
\usepackage{mathrsfs}
\usepackage{multirow}
\usepackage{txfonts}
\usepackage{color}
\usepackage{url}
\usepackage{lscape}

\DeclareMathAlphabet{\mathcal}{OMS}{cmsy}{m}{n}
\DeclareSymbolFont{largesymbols}{OMX}{cmex}{m}{n}

\objectmargin{1pt}

\newtheorem{thm}{Theorem}[section]

\newtheorem{lem}[thm]{Lemma}
\newtheorem{prop}[thm]{Proposition}
\newtheorem{cor}[thm]{Corollary}
\theoremstyle{definition}

\newtheorem{exam}[thm]{Example}
\newtheorem{rem}[thm]{Remark}
\newtheorem{subsec}[thm]{}

\newcommand{\defvariable}[2][]{
\if\relax\detokenize{#1}\relax  
   \if\relax\detokenize{#2}\relax
    \else
    ({#2})
    \fi
\else
  ^{{#1}}({#2})
\fi
 }

\newcommand{\lra}{\longrightarrow}

\newcommand{\ra}{\rightarrow}

\newcommand{\add}{{\rm add\, }}

\newcommand{\Hom}{{\rm Hom \, }}
\newcommand{\End}{{\rm End\, }}

\newcommand{\D}[1]{{\mathscr D}(#1)}

\newcommand{\Db}[1]{{\mathscr D}^b(#1)}

\newcommand{\K}[1]{{\mathscr K}(#1)}
\newcommand{\HH}[1]{{\mathscr H}(#1)}

\newcommand{\Kb}[1]{{\mathscr K}^b(#1)}
\newcommand{\Modc}{\ensuremath{\mbox{{\rm -Mod}}}}
\newcommand{\modc}{\ensuremath{\mbox{{\rm -mod}}}}

\newcommand{\op}{^{\rm op}}
\newcommand{\otimesL}{\otimes^{\rm\bf L}}

\newcommand{\Gr}{\textrm{-}\mathrm{Gr}}
\newcommand{\Mod}{\textrm{-}\mathrm{Mod}}
\newcommand{\gr}{\textrm{-}\mathrm{gr}}
\newcommand{\perf}{\textrm{-}\mathrm{perf}}
\newcommand{\Grperf}{\textrm{-}\mathrm{grperf}}

\begin{document}
\sloppy

{\Large \bf \begin{center}  Tilting complexes for group graded self-injective algebras
 \end{center}}
\medskip\medskip

\centerline{{\sc Andrei Marcus$^{a}$} and {\sc Shengyong Pan$^{ b,*}$}}

\begin{center}
Faculty of Mathematics and Computer Science, \\
Babe\c{s}-Bolyai University, \\
Cluj-Napoca, Romania\\
E-mail: marcus@math.ubbcluj.ro
\end{center}

\begin{center} $^b$ Department of Mathematics,\\
 Beijing Jiaotong University,  Beijing 100044,\\
People's Republic of China\\
E-mail: shypan@bjtu.edu.cn
\end{center}

\renewcommand{\thefootnote}{\alph{footnote}}
\setcounter{footnote}{-1} \footnote{ $^*$ Corresponding author. Email: shypan@bjtu.edu.cn}

\renewcommand{\thefootnote}{\alph{footnote}}
\setcounter{footnote}{-1} 
\renewcommand{\thefootnote}{\alph{footnote}}
\setcounter{footnote}{-1} 

\medskip
\begin{abstract}
We construct derived equivalences between group graded self-injective algebras, starting from equivalences between their $1$-components, obtained via a construction of J. Rickard and S. Al-Nofayee. \end{abstract}

\section{Introduction}

Construction of tilting complexes for group graded algebras was primarily motivated by the problem of finding reduction methods for Brou\'e's Abelian Defect Group Conjecture. In \cite{M3} two-sided tilting complexes are discussed, while in \cite{M1}, Okuyamas's work promted the need for one-sided group graded tilting complexes. The paper \cite{M2} starts from a method, due to Rickard \cite{R1}, to lift stable equivalences to derived equivalences by characterizing objects that correspond to simple modules. The context in \cite{R1} is that of symmetric algebras. This result has been generalized to self-injective algebras by Al-Nofayee \cite{Al}, and then further extended by Rickard and Rouquier \cite{RR1}.

In this  paper we obtain group graded derived equivalences between self-injective algebras starting from the main results of \cite{Al} and \cite{RR1}. Thus we generalize here the main result of \cite{M2}, and for this, we rely on the properties of the Nakayama functor in the group graded setting. In  Section \ref{s:prelim} we recall the definition and characterization  of group graded tilting complexes, and we point out the it is no need to assume the finiteness of the group $G$. Our main results in Section  \ref{s:tilting} are group graded versions of \cite[Theorem 4]{Al} and of \cite[Theorem 3.9]{RR1}. One of the applications in Section \ref{s:appl} is the combination of these results with Okuyama's strategy to lift stable equivalences. Another application is related to a construction of tilting complexes by Abe and Hoshino \cite{AH}.

In this paper, rings are associative with identity, and modules are left, unless otherwise specified. We denote by $A\Mod$ the category of left $A$-modules, and by $A\modc$ its full category consisting of finitely generated $A$-modules. If $X$ is an object of an additive category $\mathscr{A}$, $\mathrm{add}(X)$ denotes the full subcategory of $\mathscr{A}$ whose objects are direct summands of finite direct sums of copies of $X$. The notations  $\mathscr{H}(\mathscr{A})$ and $\mathscr{D}(\mathscr{A})$ stand for the (unbounded) homotopy, respectively derived category of an abelian category $\mathscr{A}$.  We freely use basic facts from \cite{M3}, \cite{M1}, \cite{M2}.

\section{Preliminaries} \label{s:prelim}

Let $k$ be a commutative ring and $G$ be a group (not necessarily finite). Suppose that $R=\bigoplus_{g\in G}R_g$ and $S=\bigoplus_{g\in G}S_g$ are $G$-graded $k$-algebras such that $R$ is $k$-flat. Throughout we denote by $A=R_1$ and $B=S_1$ the identity components of $R$ and $S$, respectively. We denote by $R\Gr$ the category of $G$-graded $R$-modules, and by $R\gr$ the category of finitely generated $G$-graded $R$-modules. 

\begin{subsec} The group $G$ acts on $G$-graded $R$-modules $M\in R\Gr$ by letting $M(g)=\bigoplus_{h\in G}M(g)_h$ be the $g$-suspension of $M$, where $M(g)_h=M_{hg}$ for all $g,h\in G$. If $R$ is strongly graded, then $G$ acts on $A$-modules $X\in A\Mod$ by conjugation $X\mapsto R_g\otimes_A X$. Note that $(R\otimes_A X)(g)$ is naturally isomorphic to  $R\otimes_A (R_g\otimes_A X)$ in $R\Gr$.
\end{subsec}

\begin{subsec} A $G$-graded $(R,S)$-bimodule $M$ can be regarded as an  $R\otimes S^{\mathrm{op}}$-module graded by the $G\times G$-set $G\times G/\delta(G)$, where $\delta(G)$ is the diagonal subgroup of $G\times G$, with $1$-component $M_1$ a module over the diagonal subalgebra \[\Delta(R\otimes S^{\mathrm{op}}):=(R\otimes S^{\mathrm{op}})_{\delta(G)}=\bigoplus_{g\in G}R_g\otimes S_{g^{-1}}.\] If $R$ and $S$ are strongly graded, then $M$ and $(R\otimes S^{\mathrm{op}})\otimes_{\Delta(R\otimes S^{\mathrm{op}})}M_1$ are naturally isomorphic $G$-graded $(R,S)$-bimodules.
\end{subsec}

\begin{subsec} Recall that an object $\tilde T$ of $\D{R\Gr}$ is called a $G$-{\it graded tilting complex} if it satisfies the following conditions:
\begin{enumerate}
\item[(i)] $\tilde T\in R\Grperf$; this means that, regarded as a complex of $R$-modules,  $\tilde T\in R\perf$, that is $\tilde T$ is bounded, and its terms are finitely generated projective $R$-modules.
\item[(ii)] $\bigoplus_{g\in G}\Hom_{\D{R\Gr}}(\tilde T,\tilde T(g)[n])=0$ for $n\neq0$.
\item[(iii)] $\mathrm{add}\{\tilde T(g)\mid g\in G\}$ generates $R\Grperf$ as a triangulated category.
\end{enumerate}
\end{subsec}

The following result was proved in {\cite[Theorem 2.4]{M1}} and \cite[Theorem 4.7]{M3}, based on Keller's approach \cite{Ke1}, but note that the assumption that $G$ is finite is not needed.

\begin{thm} The following statements are equivalent:

$(1)$ There is a $G$-graded tilting complex $\tilde T\in\D{{R}\Gr}$ and an isomorphism $S\simeq\End_{\D{R}}(\tilde T)^{\mathrm{op}}$ of $G$-graded algebras.

$(2)$ There is a complex $\tilde U$ of $G$-graded $(R,S)$-bimodules such that the functor
\[\tilde U\otimesL_S-: \D{S}\lra \D{R}\]
is an equivalence.

$(3)$ There are equivalences
\[F: \D{R}\lra \D{S} \quad and \quad F^{\mathrm{gr}}: \D{{R}\Gr}\lra \D{{S}\Gr}\]
of triangulated categories such that $F^{gr}$ is $G$-graded functor and the diagram
\[\xymatrix@M=2mm{ \D{R\Gr}\ar[d]^{\mathscr{U}} \ar[r]^{F^{\mathrm{gr}}} & \D{S\Gr}\ar[d]^{\mathscr{U}}  \\  \D{R}\ar[r]^{F} & \D{S},}\]
is commutative.

$(4)$ There are equivalences
\[F_{\mathrm{perf}}: \D{R\perf}\lra \D{S\perf} \quad and \quad F_{\mathrm{perf}}^{\mathrm{gr}}: \D{{R}\Grperf}\lra \D{{S}\Grperf}\]
of triangulated categories such that $F_{\mathrm{perf}}^{\mathrm{gr}}$ is $G$-graded functor and $\mathscr{U}\circ F_{\mathrm{perf}}^{\mathrm{gr}}=F_{\mathrm{perf}}\circ\mathscr{U}$.

{\rm(5) (provided that $R$ and $S$ are strongly graded)}  There are (bounded) complexes $U$ of $\Delta(R\otimes S^{\mathrm{op}})$ modules and  $V$ of $\Delta(S\otimes R^{\mathrm{op}})$-modules, and isomorphisms $U\otimesL_{B} V\simeq A$ in $\mathscr{D}^b(\Delta(R\otimes R^{\mathrm{op}}))$ and $V\otimesL_{A}U\simeq B$ in $\mathscr{D}^b(\Delta(S\otimes S^{\mathrm{op}}))$.
\end{thm}

\begin{proof} We have an isomorphism $S\simeq\End_{\D{R}}(\tilde T)^{\mathrm{op}}$ of $G$-graded algebras for any group $G$, because $\tilde T$ is bounded, and each component of $\tilde T$ is finitely generated. It follows that the proofs of \cite[Theorem 2.4]{M1}  and \cite[Theorem 4.7]{M3} carry over to the general situation.
\end{proof}

\begin{subsec} A complex $\tilde T\in\HH{R\Gr}$ is called $G$-{\it invariant} if $\tilde T(g)\simeq \tilde  T$ in $\HH{R\Gr}$ for all $g\in G$. More generally, $\tilde T$ is called {\it weakly $G$-invariant} if $\tilde T(g)\in\add(\tilde T)$ in $\HH{R\Gr}$ for all $g\in G$.

If $R$ is strongly $G$-graded, then a complex $T\in\HH{A}$ is called $G$-{\it invariant} if $R_g\otimes_A T\simeq T$ in $\HH{A}$ for all $g\in G$, and $T$ is called {\it weakly $G$-invariant} if $R_g\otimes_A T\in\add(T)$ in $\HH{A}$ for all $g\in G$.

Note that if $R$ is strongly $G$-graded, then the functor $R\otimes_A-: A\Modc\to R\Gr$ is an equivalence, hence a complex $T\in\HH{R\Gr}$ is  $G$-{invariant}  (weakly $G$-invariant) if and only if its identity component $T\in\HH{A}$ is  $G$-{invariant} (weakly $G$-invariant).
\end{subsec}

The following statement is also true for arbitrary $G$. It is essentially proved in \cite[Proposition 2.1 and Remark 2.2]{M2}, but for convenience, we include a proof here.

\begin{prop}\label{mm} Assume that $R$ is strongly $G$-graded. Let $\tilde T$ be a weakly $G$-invariant object in $\Kb{R\Gr}$. Denote by $T$ the identity component of $\tilde T$, and let $S=\End_{\D{R}}({{\tilde T}})^{\mathrm{op}}$.
Then, $T$ is a tilting complex for $A$ if and only if $\tilde T$ is a $G$-graded tilting complex for $R$. Moreover, in this case, $S$ is strongly $G$-graded, and it is a crossed product if and only if $\tilde T$ is $G$-invariant.
\end{prop}

\begin{proof} Since $R$ is strongly graded, the functor $R\otimes_A-: A\Modc\to R\Gr$ is an equivalence, and a $G$-graded $R$-module is projective in $R\Gr$ if and only if it is projective in $A\Modc$. It follows that $\tilde T$ is a bounded complex of finitely generated projective $R$-modules if and only if $T$ is a bounded complex of finitely generated projective $A$-modules. For each $m\in \mathbb{Z}$, we have
\[\Hom_{\K{R}}({\tilde T}, \tilde{T}[m])\simeq\bigoplus_{g\in G}\Hom_{\HH{R\Gr}}({\tilde T},{\tilde T}[m](g))\]
and
\[\Hom_{\HH{R\Gr}}({\tilde T}, {\tilde T}[m](g))\simeq \Hom_{\HH{A}}(T, R_g\otimes T[m]).\] Since $\tilde T$ is weakly $G$-invariant, we have that for $m\neq 0$, $\Hom_{\K{A}}(T, R_g\otimes T[m])=0$ if and only if $\Hom_{\K{R\Gr}}({\tilde T},{\tilde T}[m](g))=0$. If $A$ is in the triangulated subcategory generated by $\add(T)$ in $\Db{A}$, then $R$ is in the triangulated subcategory generated by $\add ({\tilde T})$ in $\Db{R\Gr}$. Conversely, if $R$ belongs to the triangulated subcategory generated by $\add (\tilde{T})$ in $\Db{R\Gr}$, then $_AR$ is in the triangulated subcategory generated by $\add (_A\tilde{T})$ in $\Db{A}$. Since $_AR$ is a finite direct sum of copies of $A$, and $_A\tilde{T}$ is a finite direct sum of copies of $T$, $A$ is in the triangulated subcategory generated by $\add (T)$ in $\Db{A}$. The last statement is clear, and also note that for the identity component of $S$ the have the isomorphism \[S_1= \End_{\HH{R\Gr}}(\tilde{T})\simeq \End_{\HH{A}}(T)\op\]
of $k$-algebras.
\end{proof}

\section{$G$-graded self-injective algebras} \label{s:tilting}

In this section we assume that $R$ is a strongly $G$-graded algebra over the field $k$, where $G$ is a finite group, and the identity component $A:=R_1$ is a finite dimensional algebra. For simplicity, we also assume that the field $k$ is algebraically closed, but the results below easily genelatize to arbitrary fields (see \cite[Section 8]{R1}).

\begin{prop}\label{mm-2} Let $T$ be a $G$-invariant object in $\HH{A}$, and denote $\tilde{T}=R\otimes_A T$ and $S=\End_{\D{R}}(\tilde{T})^{op}$.  If $T$ is a tilting complex for $A$ and $A$ is self-injective, then $S$ is a strongly $G$-graded self-injective algebra.
\end{prop}

\begin{proof} We know by Proposition \ref{mm} that  $\tilde{T}$ is a $G$-graded tilting complex for $R$ and that $S$ is strongly $G$-graded. It is easy to see that $R$ is self-injective if and only if $A$ is self-injective (see, for instance, \cite[5.1]{M3}). Finally, self-injectivity is preserved by derived equivalences by  \cite[Lemmas 1.7 and 1.8]{AH} (see also  \cite[Corollary 3.12]{RR1}).
\end{proof}

Next we extend S. Al-Nofayee's construction \cite{Al} to the case of strongly $G$-graded algebras.

\begin{subsec} \label{s-cond} If there is a derived equivalence between the self-injective $k$-algebras $A$ and $B$, then the set $\mathcal{S}=\{X_i\mid  i\in I\}$ of objects corresponding to the simple $B$-modules, must satisfy the following conditions.

(a) $\Hom(X_i,X_j[m])=0$ for $m<0$.

(b) $\Hom(X_i,X_j)=k$ if $i=j$ and $0$ otherwise.

(c) The objects $X_i$, $i\in I$, generate $\Db{A\modc}$ as a triangulated category.

(d) The Nakayama functor $\nu$ permutes the set $\mathcal{S}$, that is, there is a permutation $\sigma$ on $I$ such that $\nu(X_i)=X_{\sigma(i)}$.

In order to obtain a graded derived equivalence, we need to consider the conjugation action of  $G$ on $A$-modules. Assume that $I$ is a finite $G$-set, and that the objects $X_i$ also satisfy the condition:

{\rm(e)} $R_g\otimes_A X_i\simeq X_{gi}$ for all $g\in G$ and $i\in I$.
\end{subsec}

\begin{lem}\label{mp-2} Let $X_i\in\Db{A\modc}$, $i\in I$, be objects satisfying conditions {\rm \ref{s-cond} (a) to (e)}. There exist bounded complexes  $T_i=I_{\mathcal{S}}(X_i)$ of finitely generated injective modules,  and bounded complexes $T'_i=P_{\mathcal{S}}(X_i)$ of finitely generated projective modules, such that
\[\Hom(T_i, X_j[m])=\begin{cases}
k, &\textrm{if }  i=j\; \textrm{ and}\  m=0  \\
0, &\textrm{otherwise}.
\end{cases},\]
\[\Hom(X_j[m], T'_i)=\begin{cases}
k, &\textrm{if }  i=j\; \textrm{ and}\  m=0  \\
0, &\textrm{otherwise}.
\end{cases},\]
and moreover, \[R_g\otimes_A T_i\simeq T_{gi}, \qquad R_g\otimes_A T'_i\simeq T'_{gi},\] for all $g\in G$ and $i,j\in I$.
\end{lem}

\begin{proof} The proof given in \cite[Theorem 2.4]{M1} is based on \cite[Section 5]{R1}, and it works for self-injective algebras. For  convenience, we give a brief proof. Let $g\in G$ and $i\in I$. The construction of the complexes $T_i$ go by induction as follows.

Set $X^{(0)}_i:=X_i$, then $R_g\otimes_A X^{(0)}_i=X^{(0)}_{gi}$. By induction on $n$, we shall construct a sequence
\[X^{(0)}_i\to X^{(1)}_i\to X^{(2)}_i\to\cdots  \to X^{(n)}_i\to\cdots\]
of objects and maps in $\Db{A}$.
Assuming that $X^{(n-1)}_i$ and $X^{(n-1)}_{gi}$ are constructed such that $R_g\otimes_A X^{(n-1)}_i=X^{(n-1)}_{gi}$, we may construct that $X^{(n)}_i$ and $X^{(n)}_{gi}$  such that $R_g\otimes_A X^{(n)}_i=X^{(n)}_{gi}$ and we have the commutative diagram
\[\xymatrix{
R_g\otimes_A X^{(n-1)}_i\ar[d]\ar[r] & X^{(n-1)}_{gi} \ar[d]^{\cdot x}\\
R_g\otimes_A X^{(n-1)}_i\ar[r] & X^{(n-1)}_{gi}}.\]
Finally, let $T_i=\mathrm{hocolim}(X_i^{(n)})$, so it follows that $R_g\otimes_A T_i\simeq T_{gi}$.
\end{proof}

\begin{lem}\label{NG-commute} The permutation induced by the Nakayama functor $\nu$ commutes with the conjugation action of $G$ on $A$-modules.
\end{lem}

\begin{proof} Recall that $\nu=D\Hom_A(-,A)$, where $D=\Hom_k(-,k)$. Since for any $g\in G$, the bimodule $R_g$ induces a Morita auto-equivalence of $A$-mod with quasi-inverse $R_{g^{-1}}$, for any $g\in G$ and $i\in I$ we have
\begin{align*}
\nu (R_g\otimes_A X_i) &\simeq D\Hom_A(R_g\otimes_A X_i,A)\simeq D\Hom_A(X_i,\Hom_A(R_g, A))  \\
                     &\simeq D\Hom_A(X_i,R_{g^{-1}})\simeq D(\Hom_A(X_i, A)\otimes_A R_{g^{-1}}) \\
                     &\simeq \Hom_A(R_{g^{-1}}, D\Hom_A(X_i, A) )\simeq \Hom_A(R_{g^{-1}}, \nu(X_i))   \\
                     &\simeq \Hom_A(R_{g^{-1}}, A)\otimes \nu(X_i))\simeq R_g\otimes_A \nu(X_i) \\
                     &\simeq X_{g\sigma(i)},
\end{align*}
On the other hand,
\[\nu (R_g\otimes_A X_i)\simeq \nu(X_{gi})\simeq X_{\sigma(gi)},\]
so $\sigma(gi)=g\sigma(i)$ for all $g\in G$ and $i\in I$.
\end{proof}

\begin{thm} \label{mp} Let $R$ be a strongly $G$-graded self-injective algebra with $R_1=A$, let $I$ be a finite $G$-set, and let $X_i\in\Db{A\modc}$, $i\in I$, be objects satisfying conditions {\rm \ref{s-cond} (a) to (e)}.

Then there is another self-injective crossed product $G$-algebra $S$, and a $G$-graded derived equivalence between $R$ and $S$, whose restriction to $A$ sends $X_i, i\in I$, to the simple $S_1$-modules.
\end{thm}

\begin{proof}  By  \cite[Lemma 5]{Al}, there is a tilting complex $T=\bigoplus_{i\in I}T_i$ for $A$ such that
\[\Hom(T_i, X_j[m])=\begin{cases}
k, &\textrm{if }  \sigma(i)=j\; \textrm{ and}\  m=0,  \\
0, &\textrm{otherwise}.
\end{cases}\]
It follows by Lemma \ref{NG-commute} and by the definition of the homotopy colimit that $\nu(T_i)\simeq T_{\sigma(i)}$ for all $i\in I$ (see \cite[Lemma 9]{Al}). By Lemma \ref{mp-2}, the summands $T_i$ can be constructed to satisfy the additional condition
\[R_g\otimes_A T_i\simeq T_{gi}.\]
for all $i\in I$ and $g\in G$. Consequently, $T$ is $G$-invariant, and Proposition \ref{mm} applies
\end{proof}

\begin{subsec}  \label{s:t-structure}


Let $\mathcal{T}=\mathcal{D}^b(A)$, and let $(\mathcal{T}^{\le 0}, \mathcal{T}^{>0})$ be the bounded $t$-structure on $\mathcal{T}$ as in \cite[Proposition 3.4]{RR1}. Denote by $\mathcal{A}$ the heart of this $t$-structure, and by ${}^tH^0$ the $H^0$-functor associated to this $t$-structure. Then the set of the simple objects of $\mathcal{A}$ is $\mathcal{S}$.  Let $T_i=I_{\mathcal{S}}(X_i)$ and $T'_i=P_{\mathcal{S}}(X_i)$, $i\in I$, be the complexes defined in Lemma \ref{mp-2}, and let $T'=\bigoplus_{i\in I}T'_i$.

Consider the finite dimensional $G$-graded DG algebra (see \cite[2.3]{M1})
\[S = \End^\bullet_R(R\otimes_A \bigoplus_{i\in I} P_{\mathcal{S}}(X_i))\] 
with $1$-component
\[B = \End^\bullet_A( \bigoplus_{i\in I} P_{\mathcal{S}}(X_i)) = \bigoplus_m \Hom_A( \bigoplus_{i\in I} P_{\mathcal{S}}(X_i), \bigoplus_{i\in I} P_{\mathcal{S}}(X_i)[m]).\]
\end{subsec}

We may now extend \cite[Theorem 3.9]{RR1} to strongly $G$-graded algebras.

\begin{thm} We have:

{\rm1)} $H^m(S) =0$ for $m> 0$ and for $m\ll 0$.

{\rm2)} There is a $G$-graded derived equivalence  $\Db {S\modc} \simeq \Db {R\modc}$.

{\rm3)} There is a $G$-equivalence $H^0(B)\textrm{-}\mathrm{mod} \simeq \mathcal{A}$.
\end{thm}

\begin{proof} 1) By \cite[Theorem 3.9]{RR1}, we have that  $H^m(B) =0$ for $m> 0$ and for $m\ll 0$. Note that
\[\Hom_R(R\otimes_A P_{\mathcal{S}}(X_j), R\otimes_A P_{\mathcal{S}}(X_i)[m])\simeq \Hom_A(P_{\mathcal{S}}(X_j), \bigoplus_{g\in G}R_g\otimes_A P_{\mathcal{S}}(X_i)[m]),\]
and, by Lemma \ref{mp-2}, $\Hom(P_{\mathcal{S}}(X_j), R_g\otimes_A P_{\mathcal{S}}(X_i)[m])=0$ for all $i,j\in I$  if and only if $\Hom(P_{\mathcal{S}}(X_j), P_{\mathcal{S}}(X_i)[m])=0$ for all $i,j\in I$. Consequently, $H^i(S) =0$ for $m> 0$ and for $m\ll 0$.

2) We also know that the functor
\[\Hom^\bullet_A( \bigoplus_{i\in I}P_{\mathcal{S}}(X_i),-): \Db{A}\to \Db{B}\]
is an equivalence. Since $R\otimes_A-: A\Modc\to R\Gr$ is an equivalence, and a $G$-graded $R$-module is projective in  $R\Gr$ if and only if it is projective in $A\Modc$, it is clear that $R_A\otimes_A(\bigoplus_{i\in I} P_{\mathcal{S}}(X_i))$ is perfect object in $\D{R\Gr}$ if and only if $\bigoplus_{i\in I} P_{\mathcal{S}}(X_i)$ is a perfect complex of $A$-modules. Therefore, we get the $G$-graded derived equivalence
\[\Hom^\bullet_R( R\otimes_A\bigoplus_{i\in I} P_{\mathcal{S}}(X_i),-): \Db{R\modc}\to \Db{S\modc}.\]

3) By \cite[Theorem 3.9]{RR1}, ${}^tH^0(T')$ is a progenerator for $\mathcal{A}$ with endomorphism algebra $H^0(B)$. As in Lemma \ref{mp-2}, $T'$ is $G$-invariant, hence  ${}^tH^0(T')$ is also $G$-invariant, and the statement follows.
\end{proof}

\section{Applications and examples} \label{s:appl}

Okuyama's  strategy to lift a stable equivalence to a derived equivalence also generalizes to strongly $G$-graded self-injective algebras. We assume that $k$ is a field, and here we need to assume in addition that the order of $G$ is invertible in $k$, that is,  the characteristic of $k$ does not divide $|G|$.

\begin{cor} Let $R$ and $S$ be strongly $G$-graded self-injective algebras. Assume that $|G|$ is invertible in $k$, and let $M$ be a $G$-graded $R$-$S$-bimodule inducing a stable equivalence of Morita type between $R$ and $S$.

If there are objects $X_i\in\Db{A\modc}$, $i\in I$, be objects satisfying conditions {\rm \ref{s-cond} (a) to (e)}, and such that $X_i$ is stably isomorphic to $M_1\otimes_B S_i$, for all $i\in I$, then there is a $G$-graded derived equivalence between $R$ and $S$.
\end{cor}

\begin{proof} By Theorem \ref{mp} there is a self-injective crossed product $R'$ and a  $G$-graded derived equivalence between  $R$ and  $R'$ By \cite[Remark 3.4]{M1}, we obtain a  $G$-graded stable equivalence of Morita type between  $R$ and  $R'$. Consequently, we have a stable Morita equivalence between  $R'$ and $S$ induced by a  $G$-graded $R'\otimes S^{\mathrm{op}}$-bimodule $M'$. Since simple $A'$-modules are sent to simple $B$-modules, a theorem of Linckelmann \cite[Theorem 2.1]{Li} says that a direct $A\otimes B^{\mathrm{op}}$-summand $M$ of $M'_1$ induces a Morita equivalence between $A'$ and $B$. Since $|G|$ is invertible in $k$, we have that $M$ is a $\Delta(R'\otimes S^{\mathrm{op}})$-summand of $M'_1$, so by \cite[Theorem 3.4]{M3}, $(R'\otimes S^{\mathrm{op}})\otimes_{\Delta(R'\otimes S^{\mathrm{op}})}M$ induces a $G$-graded Morita equivalence between $R'$ and $S$. By composing this equivalence with the $G$-graded derived equivalence between $R$ and $R'$, we obtain a $G$-graded derived equivalence between $R$ and $S$.
\end{proof}

\begin{rem} By \cite[Section 4]{Al} and \cite[Section 6.3]{R1}, we must have
\[X_i\simeq \Omega^{n_i}(M_1\otimes_B S_i)[n_i].\]
Here we only have to verify condition \ref{s-cond} (e). But this follows immediately since $R_g\otimes_AM\otimes_BS_{g^{-1}}\simeq M$ as $(A,B)$-bimodules, and the syzygy functor $\Omega$ also commutes with the $G$-conjugation functor $R_g\otimes_A-$.
\end{rem}

Another construction of tilting complexes for self-injective algebras is given in \cite[Section 3]{AH} in the case of representation-finite algebras. Here we adapt \cite[Theorem 3.6]{AH} in order to obtain a $G$-invariant tilting complex, so that Proposition \ref{mm} can be applied. Here $G$ is not assumed to be finite.

\begin{prop} Assume that $A$ is representation-finite. Let $P$  be a bounded complex of finitely generated projective $A$-modules such that $\Hom_{\HH{A}}(P, R_g\otimes_A P[m]) = 0$ for all $m\neq 0$ and $g\in G$, and $\add P= \add \nu P$. Then there exists a bounded complex of finitely generated projective $A$-modules $Q$ such that $Q\oplus P$ is a $G$-invariant tilting complex.
\end{prop}

\begin{proof} By adding $G$-conjugates of $P$, we may assume that $P$ is $G$-invariant. The complex $P$ defines a torsion theory $(\mathscr{T},\mathscr{F})$, where  $\mathscr{T}={}^\perp H^0(P)$ and $\mathscr{F}=\mathscr{T}^\perp$ are invariant under the Nakayama functor $\nu$. Since $P$ is $G$-invariant, we have that $\mathscr{T}$ and $\mathscr{F}$ are also closed under $G$-conjugation. Let $\{e_j\mid j\in J\}$ be a basic set of orthogonal local idempotents in $A$. The $G$-conjugation action of $G$ on $A$-modules induces a $G$-set structure on $J$, such that then the subsets $J_1=\{j\in J\mid Ae_j\in \mathscr{T}\}$ and $J_2=\{j\in J\mid Ae_j\in \mathscr{F}\}$ are $G$-stable. On can easily deduce from these observations and \ref{NG-commute} that the tilting complex $T$ constructed in \cite[Lemma 3.4]{AH} and the complex $Q$ from the proof of \cite[Theorem 3.6]{AH} are $G$-invariant.
\end{proof}

\begin{exam} This example  is related to \cite[Example 9.5]{A}. Let $A$ be a finite dimensional $k$-algebra given by the quiver
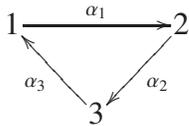
\[\xymatrix{     1 \ar[rr]^{\alpha_1} & & 2 \ar[dl]^{\alpha_2} \\
& 3 \ar[ul]^{\alpha_3}               }\]
with relations \[\alpha_1\alpha_2\alpha_3\alpha_1=\alpha_2\alpha_3\alpha_1\alpha_2=\alpha_3\alpha_1\alpha_2\alpha_3=0.\] Let $P=P_2\oplus P_3$. Then $\add P=\add\nu(P)$. Then there is a complex \[Q:=0\ra P_2\ra P_1\ra 0\] with $P_2$ in degree $0$ such that $P\oplus Q$ is a tilting complex for $A$, and the endomorphism algebra  $B$ of $P\oplus Q$ is the algebra given by the quiver
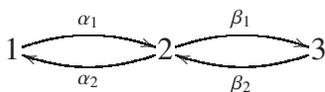
\[\xymatrix{ 1 & & 2 & & 3
\ar@/^/^{\alpha_1}"1,1";"1,3"
\ar@/^/^{\beta_1}"1,3";"1,5"
\ar@/^/^{\alpha_2}"1,3";"1,1"
\ar@/^/^{\beta_2}"1,5";"1,3"     }\]
with relations \[\alpha_1\beta_1=\beta_2\alpha_2=0, \ \alpha_1\alpha_2\alpha_1=\beta_2\beta_1\beta_2=\alpha_2\alpha_1-\beta_1\beta_2=0.\]
Moreover, consider  the infinite cyclic group $G=\langle g, g^{-1}\mid gg^{-1}=g^{-1}g=1\rangle$ acting on $A$ and $B$ as follows: The element $g$ fixes all the vertices and the edge $\alpha_1$ in $A$,  and $g(\alpha_i)=\alpha_i+\alpha_i\alpha_{i+1}\alpha_{i+2}\alpha_i \pmod 3$  for $i\neq 1$,  while $g$ fixes all the vertices and all $\alpha_i$ in $B$, and $g(\beta_i)=\beta_i+\beta_i\beta_{i+1}\beta_i \pmod 2$  for $i\neq 1$.
Then the complexes $P$ and $Q$ are $G$-invariant, and $R\otimes_A(P\oplus Q)$ induces  a $G$-graded derived equivalence between the skew group algebras $R=A*G$ and $S=B*G$.
\end{exam}

\begin{exam} Let $A$ be a finite dimensional $k$-algebra given by the quiver
\[\xymatrix{     1 \ar[r]^{\alpha_1} & 3 & 2 \ar[l]_{\alpha_2}  }\]
(see \cite[1.4]{RR} and \cite[Example (2)]{Mi}). Let $G=\{1,g\}$, with $g$ acting on $A$ by interchanging vertices $1$ and $2$, and consider the skew group algebra $R=A*G$.  Let $T=P_1\oplus P_2\oplus I_3$. Then $R\otimes_A T$ is a $G$-graded tilting complex for $R$.
\end{exam}

\bigskip

\noindent{\bf Acknowledgements.} Shengyong Pan is funded by China Scholarship Council.

The authors thank the referee for his/her observations which improved the presentation of the paper.

{\footnotesize

\end{document}